\documentclass[11pt]{amsart}
\usepackage{amsfonts, amsmath, amssymb, amsthm}
\usepackage{epic}
\usepackage{tikz}
\usepackage{tikz-cd}
\usepackage{graphicx,color}
\usepackage{changebar}
\usetikzlibrary{matrix}

\newtheorem{proposition}{Proposition}[section]

\newtheorem{lemma}[proposition]{Lemma}
\newtheorem{theorem}[proposition]{Theorem}

\newtheorem{corollary}[proposition]{Corollary}
\newtheorem{definition}[proposition]{Definition}

\newtheorem{remark}[proposition]{Remark}
\newtheorem{fact}[proposition]{Fact}

\newcommand{\nth}[1]{$#1 \hbox{-th }$}

\newcommand{\ord}{{\rm ord} }

\newcommand{\rem}{{\rm rem } }

\newcommand{\Z}{\mathbb{Z}}
\newcommand{\ZM}[1]{\Z /( #1 \cdot \Z)}
\newcommand{\ZMs}[1]{(\Z / #1 \cdot \Z)^*}

\newcommand{\rg}[1]{\mathbf{#1}}
\newcommand{\eu}[1]{\mathfrak{#1}}

\newcommand{\id}[1]{\mathcal{#1}}
\newcommand{\Gal}{\hbox{Gal}}

\newcommand{\Gals}{\hbox{\tiny Gal}}

\newcommand{\Ker}{\hbox{ Ker }}

\newcommand{\prk}{p\hbox{-rk} }

\newcommand{\rf}[1]{(\ref{#1})}

\newcommand{\apr}[1]{{\rm \raise1ex\hbox{\tiny p}\kern-.1667em\hbox{A}}^-(#1)}
\newcommand{\apro}{{\rm \raise1ex\hbox{\tiny p}\kern-.1667em\hbox{A}}}

\newcommand{\Norm}{\mbox{\bf N}}
 
\newcommand{\lchooses}[2]{\left( \frac{#1}{#2 } \right)}

\newcommand{\F}{\mathbb{F}}

\newcommand{\A}{\mathbb{A}}

\newcommand{\B}{\mathbb{B}}
\newcommand{\E}{\mathbb{E}}
\newcommand{\K}{\mathbb{K}}

\newcommand{\KH}{\mathbb{H}}

\newcommand{\KL}{\mathbb{L}}

\newcommand{\M}{\mathbb{M}}
\newcommand{\Q}{\mathbb{Q}}

\newcommand{\N}{\mathbb{N}}

\newcommand{\T}{\mathbb{T}}

\def\ra{\rightarrow}

\def\wh{\widehat}

\begin{document}
{\obeylines \small
\vspace*{-1.0cm}
\hspace*{3.5cm}C'est pour toi que je joue, Alf c'est pour toi,
\hspace*{3.5cm}Tous les autres m'\'ecoutent, mais toi tu m'entends ...
\hspace*{3.5cm}\'Exil\'e d'Amsterdam vivant en Australie,
\hspace*{3.5cm}Ulysse qui jamais ne revient sur ses pas ... 
\hspace*{3.5cm}Je suis de ton pays, m\'et\'eque comme toi, 
\hspace*{3.5cm}Quand il faudra mourir, on se retrouvera\footnote{Free after Georges Moustaki, ``Grand-p\`ere''}
\vspace*{0.5cm}
\hspace*{5.5cm} {\it To the memory of Alf van der Poorten}
\vspace*{0.5cm}
\smallskip
}
\title[Vanishing of $\mu$ ] {On the vanishing of Iwasawa's constant
  $\mu$ for the cyclotomic $\Z_p$-extensions of $CM$ number
  fields.}
\author{Preda Mih\u{a}ilescu} \address[P. Mih\u{a}ilescu]{Mathematisches Institut
 der Universit\"at G\"ottingen}
\email[P. Mih\u{a}ilescu]{preda@uni-math.gwdg.de}
\date{Version 2.0 \today}
\vspace{0.5cm}
\begin{abstract}
  We prove that $\mu = 0$ for the cyclotomic $\Z_p$-extensions of CM
  number fields.
 \end{abstract}
\maketitle
\tableofcontents
\section{Introduction}
Iwasawa gave in his seminal paper \cite{Iw1} from $1973$ examples of
$\Z_p$-extensions in which the structural constant $\mu \neq 0$. In
the same paper, he proved that if $\mu = 0$ for the cyclotomic
$\Z_p$-extension of some number field $\K$, then the constant vanishes
for any cyclic $p$-extension of $\K$ -- and thus for any number field
in the pro-$p$ solvable extension of $\K$. Iwasawa also suggested in
that paper that $\mu$ should vanish for the cyclotomic
$\Z_p$-extension of all number fields, a fact which is sometimes
called \textit{Iwasawa's conjecture}. The conjecture has been proved
by Ferrero and Washington \cite{FW} for the case of abelian fields. In
this paper, we give an independent proof, which holds for all CM
fields:
\begin{theorem}
\label{main}
Let $\K$ be a CM number field and $p$ an odd prime.  Then Iwasawa's
constant $\mu$ vanishes for the cyclotomic $\Z_p$-extension
$\K_{\infty}/\K$.
\end{theorem}

\subsection{Notations and basic facts on decomposition of $\Lambda$-modules}
In this paper $p$ is an odd prime. In the sequel, the Iwasawa constant
$\mu$ for some number field $\K$ will always refer to the
$\mu$-invariant for the $\Z_p$-cyclotomic extension of $\K$. We shall
denote number fields by black board bold characters, e.g. $\K, \KL$,
etc. If $\K$ is a number field, its cyclotomic $\Z_p$-extension is
$\K_{\infty}$ and $\B_{\infty}/\Q$ is the $\Z_p$-extension of $\Q$, so
$\K_{\infty} = \K \cdot \B_{\infty}$.  We denote as usual by $\Gamma$ the
Galois group $\Gal(\K_{\infty}/\K)$ and let $\tau \in \Gamma$ be a
topological generator.

Let $\B_1 = \Q$ and $\B_n \subset \B_{\infty}$ be the intermediate
extensions of $\B_{\infty}$ and let the counting be given by $[ \B_n : \Q ] =
p^{n-1}$, so $\B_n = \B_{\infty} \cap \Q[ \zeta_{p^n} ]$.  Let $\kappa
> 0$ be the integer for which $\K \cap \B_{\infty} = \B_{\kappa}$;
then $\mu_{p^{\kappa}} \subset \K$ and $\mu_{p^{\kappa+1}} \not
\subset \K$. We shall use a counting of the intermediate fields of
$\K_{\infty}$ that reflects this situation and let $\K_0 = \K_1 =
\ldots = \K_{\kappa} = \K$ and $[ \K_{\kappa+1} : \K ] = p$, etc. The
constant $\kappa$ can be determined in the same way for any number
fields and we adopt the same counting in any cyclotomic
$\Z_p$-extension.  This way, for $n \geq \kappa$ we always have
$\mu_{p^n} \subset \K_n$. We let $\gamma \in \Gal(\B_{\infty}/\Q)$ be
a topological generator of $\Gal(\B_{\infty}/\Q)$ and let $\tau \in
\Gamma := \Gal(\K_{\infty}/\K)$ be a topological generator for
$\Gamma$.  We may thus assume that $\tau = \gamma^{p^{\kappa-1}}$ for
some lift $\gamma \in \Gal(\K_{\infty}/\Q)$ of $\gamma$ and write as
usual $T = \tau-1, \Lambda = \Z_p[[ T ]]$ and
\begin{eqnarray*}
 \omega_n & = & \tau^{p^{n-\kappa}} - 1 = (T+1)^{p^{n-\kappa}} - 1, \\
 \nu_{m,n} & = & \omega_m/\omega_n, \quad \hbox{for $m > n \geq \kappa$}.
\end{eqnarray*}
Since $\K = \K_1 = \K_{\kappa}$, we may also write $\nu_{m, 1} =
\nu_{m, \kappa}$.  Note that the special numeration of fields which we
introduce in order to ascertain that $\mu_{p^n} \subset \K_n$ induces
a shift in the exponents in the definition of $\omega_n$. In terms of
$\gamma$ we recover the classical definitions, but our
exposition is made in terms of the topological generator $\tau$.
Note that the base field $\K$ with respect to which we shall
bring our proof, is still to be defined, and will be a CM on
in which we assume that $\mu > 0$, and in which some useful
additional conditions are fulfilled.

The $p$-Sylow subgroups of the class groups $\id{C}(\K_n)$ of the
intermediate fields of $\K_n$ are denoted as usual by $A_n = A_n(\K) =
\left( \id{C}(\K_n) \right)_p$; the explicit reference to the base
field $\K$ will be used when we refer simultaneously to sequences of
class groups related to different base fields. Traditionally, $A =
A(\K) = \varinjlim A_n$. We use the projective limit, which we 
denote by $\apr{\K} = \varprojlim(A_n)$, as explained in detail below. 
The maximal $p$-abelian unramified
extensions of $\K_n$ are denoted by $\KH_n = \KH_n(\K) \supset \K_n$
and $X_n = \Gal(\KH_n/\K_n)$.  The projective limit with respect to
restriction maps is $X = \varprojlim(X_n)$: it is a Noetherian
$\Lambda$-torsion module on which complex conjugation acts inducing
the decomposition $X = X^+ \oplus X^-$.  The Artin maps $\varphi : A_n
\ra X_n$ are isomorphisms of finite abelian $p$-groups; whenever the
abelian extension $\M/\K$ is clear in the context, we write
$\varphi(a)$ for the Artin Symbol $\lchooses{\M/\K}{a}$, where $a \in
\id{C}(\K)$ if $\M$ is unramified, or $a$ is an ideal of $\K$
otherwise.  Complex conjugation acts on class groups and Galois
groups, inducing direct sum decompositions in plus and minus parts:
\[ A_n = A_n^+ \oplus A_n^-, \quad X_n = X_n^+ \oplus X_n^-, \quad
\hbox{etc.} \] The idempotents generating plus and minus parts are
$\frac{1 \pm \jmath}{2}$; since $2$ is a unit in the ring $\Z_p$
acting on these groups, we also have $Y^+ = (1+\jmath) Y$ and $Y^- =
(1-\jmath) Y$, for $Y \in \{ A_n, X_n, X, \ldots \}$.  Throughout the
paper, we shall use, by a slight abuse of notation and unless
explicitly specified otherwise, the additive writing for the group
ring actions. This is preferable for typographical reasons.  If $M$ is
some Noetherian $\Lambda$-torsion module on which $\jmath$ acts,
inducing a decomposition $M = M^+ \oplus M^-$, we write $\mu^-(M) =
\mu(M^-), \lambda^-(M) = \lambda(M^-)$, etc, for the Iwasawa constants
of this module. In the case when $M = \apr{\M}$ or $M = X(\M)$ is
attached to some number field $\M$, we simply write $\mu^-(\M)$ or
$\mu(\apr{\M})$.

By assumption on $\K$, the norms $\Norm_{\K_m/\K_n} : A_m \ra A_n$ are
surjective for all $m > n > 0$.  Therefore, the sequence $(A_n)_{n \in
  \N}$ is projective with respect to the norm maps and we denote their
projective limit by $\apro = \varprojlim_n A_n$. The Artin map induces
an isomorphism of compact $\Lambda$-modules $\varphi :\ \apro \ra
X$. The elements of $\apro$ are norm coherent sequences $a = (a_n)_{n
  \in \N} \in \apro$ with $a_n \in A_n$ for $n \geq \kappa$; we let
$a_0 = a_1 = \ldots = a_{\kappa}$.  It is customary to identify $X$
with $\apro$ via the Artin map; we shall not use injective limits
here, but make explicit reference to the projective limit $\apro$.

It is a folklore fact that if $\mu$ vanishes for the cyclotomic
$\Z_p$-extension of some CM field $\K_{start}$, then it vanishes for
any finite algebraic extension thereof. We prove this in Fact
\ref{blow} in the Appendix. In order to prove the Theorem \ref{main}
we shall need to taylor some base field $\K/\K_{start}$, which is a
Galois CM extension of $\Q$ and enjoys some additional conditions that
shall be discussed bellow; of course, we assume that $\mu(\K) >
0$. Before describing the construction of $\K$ however, we need to
introduce some definitions and auxiliary constructions.

\subsection{Decomposition and Thaine Shifts}
Let $M$ be a Noetherian $\Lambda$-torsion module. It is associated to
an \textit{elementary} Noetherian $\Lambda$-torsion module $\id{E}(M)
\sim M$ defined by:
\begin{eqnarray*}
\begin{array}{c c c c c c c }
  \id{E}(M) & = & \id{E}_{\lambda}(M) & \oplus & \id{E}_{\mu}(M), & \quad & \hbox{with}   \\
  \id{E}_{\mu}(M) & = & \oplus_{i=1}^r \Lambda/(p^{e_i} \Lambda), & \quad & \id{E}_{\lambda}(M) & = & \oplus_{j=1}^{r'} 
  \Lambda/(f_j^{e'_j} \Lambda),  
\end{array}
\end{eqnarray*}
where all $e_i, e'_j > 0$ and $f_j \in \Z_p[ T ]$ are irreducible
distinguished polynomials.  The pseudoisomorphism $M \sim \id{E}(M)$
is given by the exact sequence
\begin{eqnarray}
\label{psis}
  1  \ra  K_1  \ra  M  \ra  \id{E}(M)  \ra  K_2  \ra  1 ,  
\end{eqnarray}
in which the kernel and cokernel $K_1, K_2$ are finite. We define
$\lambda$- and $\mu$-parts of $M$ as follows:
\begin{definition}
  Let $M$ be a Noetherian $\Lambda$-torsion module. The $\lambda$-part
  $\id{L}(M)$ is the maximal $\Lambda$-submodule of $M$ of finite
  $p$-rank. The $\mu$-part $\id{M}(M)$ is the $\Z_p$-torsion submodule
  of $M$; it follows from the Weierstra{\ss} Preparation Theorem that
  there is some $m > 0$ such that $\id{M}(M) = M[ p^m ]$. The maximal
  finite $\Lambda$-submodule of $M$ is $\id{F}(M)$, its finite
  part. By definition, $\id{L}(M) \cap \id{M}(M) = \id{F}(M)$.

  Let the module $\id{D}(M) = \id{L}(M) + \id{M}(M)$ be the {\em
    decomposed submodule} of $M$.  Then for all $x \in \id{D}(M)$
  there are $x_{\lambda} \in \id{L}(M), x_{\mu} \in \id{M}(M)$ such
  that $x = x_{\lambda} + x_{\mu}$, the decomposition being unique iff
  $\id{F}(M) = 0$.  The pseudoisomorphism $M \sim \id{E}(M)$ implies
  that $[ M : \id{D}(M) ] < \infty $.

  If $x \in M \setminus \id{D}(M)$, the $L$- and the $D$-orders of $x$
  are, respectively
\begin{eqnarray}
\label{dords}
\ell(x) & = & \min\{  j > 0 \ : \ p^{j} x \in \id{L}(M) \}, \quad \hbox{and} \\
\delta(x) & = & \min\{  k > 0 \ : \ p^{k} x \in \id{D}(M) \} \leq \ell(x). \nonumber
\end{eqnarray}

We say that a Noetherian $\Lambda$-module $M$ of $\mu$-type is {\em
  rigid} if the map $\psi : M \ra \id{E}(M)$ is injective. Rigid
modules have the fundamental property that for any distinguished
polynomial $g(T) \in \Z_p[ T ]$ and any $x \in M$
\begin{eqnarray}
\label{rigid}
g(T) x = 0 \quad \Leftrightarrow \quad x = 0.
\end{eqnarray}
\end{definition}
Note that for CM base fields $\M$, the modules $\id{M}(\apr{\M_{\infty}})$,
where $\M_{\infty}$ is the cyclotomic $\Z_p$-extension, are rigid.  The
following fact about decomposition is proved in the last section of
the Appendix.
\begin{proposition}
\label{tpdeco}
Let $\M$ be a number field, let $\T_n \supset \KH(\M_n) \supset \M_n$
be the ray class fields to some fixed ray $\eu{M}_0 \subset
\id{O}(\M)$ and $M = \apr{\M}$ be the projective limit of the galois
groups $T_n = \Gal(\T_n/\M_n)$. Assume in addition that the following
condition is satisfied by $\M$: if $r = \prk(\id{L}(M))$ and $L_1 =
\Norm_{\M_{\infty}/\M_1}(\id{L}(M))$ then
\begin{eqnarray}
\label{cstab}
\prk(L_1) = r \quad \hbox{ and $\ord(x) > p^2$ for all $x \in L_1 \setminus L_1^p$}.
\end{eqnarray}
If these hypotheses hold and $x \in M$ is such that $p x \in
\id{D}(M)$, then $T^2 x \in \id{D}(M)$.
\end{proposition}

We observe that the condition \rf{cstab} can be easily satisfied by
eventually replacing an initial field $\M$ with some extension, as explained
in Fact \ref{b2} of Appendix 3.1 

Next we define Thaine shifts and lifts.  Let $\M$ be a CM Galois field
and consider $a = (a_n)_{n \in \N} \in \apr{\M}$, some norm coherent
sequence. We assume that the norm maps $\Norm_{\M_n/\M_{n'}} :
A_n^-(\M) \ra A_{n'}^-(\M)$ are surjective for all $n > n' \geq 1$ and
fix some integer $m > \kappa(\M)$ -- with $\M \cap \B_{\infty} =
\B_{\kappa(\M)}$ -- and a totally split prime $\eu{q} \in a_{m}$,
which is inert in $\M_{\infty}/\M_m$. Let $q \in \Q$ be the rational
prime below $\eu{q}$ and assume that $q \equiv 1 \bmod p$.  Let $\F
\subset \Q[ \zeta_q ]$ be the subfield of degree $p$ in the \nth{q}
cyclotomic extension. We let $\KL_n = \M_n \cdot \F$ and $\KL_{\infty}
= \M_{\infty} \cdot \F$.  The tower $\KL_{\infty}/\KL$ is {\em the
  inert Thaine shift} of the initial cyclotomic extension
$\M_{\infty}/\M$, induced by $\eu{q} \in a_m$.  Let $\eu{Q} \subset
\KL_m$ be the ramified prime above $\eu{q}$. According to Lemma
\ref{thlift} in the Appendix, we may apply Tchebotarew's Theorem in
order to construct a sequence $b = (b_n)_{n \in \N} \in \apr{\KL}$
such that $b_m = [ \eu{Q} ]$ is the class of $\eu{Q}$ and
$\Norm_{\KL_n/\M_n}(b_n) = a_n$ for all $n \in \N$. In the projective
limit, we then also have $\Norm_{\KL_{\infty}/\M_{\infty}}(b) = a$.  A
sequence determined in this way will be denoted {\em a Thaine lift of
  $a$}. It is not unique.

Let $F = \Gal(\F/\Q)$ be generated by $\nu \in F$; we write $s := \nu
- 1$ and $\Phi_p(\nu) = \frac{(s+1)^p - 1}{s}$. By using the identity
$\frac{x^p-1}{x-1} = \frac{(y+1)^p - 1 }{y} = y^{p-1} + O(p)$, we see
that the algebraic norm verifies
\begin{eqnarray}
\label{anorm}
\id{N} & := & \sum_{i=0}^{p-1} \nu^i = \Phi_p(\nu) = p u(s) + s^{p-1} = 
p + s f(s), \\
& & \quad \nonumber  f \in \Z_p[ X ] \setminus p \Z_p[ s ],  \quad \in (\ZM{p^N}[ s ])^{\times}, \forall N > 0. 
\end{eqnarray}

Since $\eu{q}$ is totally split in $\M_m/\Q$, the extensions
$\KL_n/\Q$ are Galois and $F$ commutes with $\Gal(\M_m/\Q)$, for every
$m$ and $\Gal(\KL_n/\M_n) \cong F$.

\subsection{Constructing the base field}
For our proof we shall choose a base field $\K$ as follows. Start with
some CM field $\K_{start}$ for which one assumes that $\mu > 0$, and
let $-D$ be a quadratic non-residue modulo $p$ -- so
$\lchooses{-D}{p} = -1$ -- and $\rg{k} = \Q[ \sqrt{-D} ]$ be a
quadratic imaginary extension. Our start field should be galois
and contain the \nth{p} roots of unity.  We require thus that $\K
\supset \K_{start}^{(n)}[ \zeta_p, \sqrt{-D} ]$.

We additionally expect that the primes that ramify in $\K_{\infty}/\K$
be totally ramified and the norms $N_{n,m} : A(\K_n) \ra A(\K_m)$ be
surjective, so $\K_{\infty} \cap \KH(\K) = \K$. We also require that
the exponent $\exp(\id{M}(\apr{\K})) \geq p^2$, which can be achieved
by means of Fact \ref{b2}. The first step of the construction consists thus in
replacing $\K_{start}$ by an initial field $\K_{ini} =
\K_{start}^{(n)}[ \zeta_p, \sqrt{-D} ]$.  If
$\exp(\id{M}(\apr{\K_{ini}})) = p$, then replace $\K_{ini}$ by some
Thaine shift thereof, in order to increase the exponent. We need to
fulfill the condition \rf{cstab} required in Proposition
\ref{tpdeco}; for this we determine an integer $t$ as follows: let $L
= \id{L}(\apr{\K_{ini}})$ and $r = \prk(L) = \lambda(\K_{ini})$.  Let
$L_t = \Norm_{\K_{ini,\infty}}/\K_{ini; t}(L)$. We then require that
\begin{eqnarray}
\label{stab}
\prk(L_t) = r \quad \hbox{ and $\ord(x) > p^2$} \quad \hbox{ for all $x \in L_t \setminus L_t^p$}.
\end{eqnarray}

\begin{figure}
\centering
\begin{tikzpicture}[node distance = 2cm, auto]
      \node (Q) {$\mathbb{Q}$};
      \node (k) [above of=Q, right of=Q,node distance = 2cm] {$\rg{k} =\mathbb{Q}(\sqrt{-D})$};
      \node (Ks) [below of=Q, right of=Q, right of=Q] {$\mathbb{K}_{start}$};
      \node (Ksp) [above of=Ks, right of=Ks, node distance = 2cm] {$\mathbb{K}_{ini}=\mathbb{K}_{start}^{(n)} (\zeta_p, \sqrt{-D})$};
      \node (Kini) [above of=Ksp, right of=Ksp, node distance = 1cm] {$\mathbb{K}'_{ini}=\mathbb{K}_{ini, t}$};
      \node (Bk) [above of=Q] {$\mathbb{B}_{\kappa}$};
      \node (B) [above of=Bk, node distance = 3cm] {$\mathbb{B}$};
      \node (kk) [above of=k] {$k_{\kappa}$};
      \node (kinf) [above of=kk, node distance = 3cm] {$k_{\infty}$};
      \node (K) [above of=Kini] {$\mathbb{K}=\mathbb{K}'_{ini,k'} = \mathbb{K}_{ini,k'+t}$};
      \node (Kinf) [above of=K, node distance = 3cm] {$\mathbb{K}_{\infty}$};
      \draw[-] (Q) to node {} (k);
      \draw[-] (Q) to node {} (Ks);
      \draw[-] (k) to node {} (Ksp);
      \draw[-] (Ks) to node {} (Ksp);
      \draw[-] (Q) to node {} (Bk);
      \draw[->] (Bk) to node {} (B);
      \draw[-] (k) to node {} (kk);
      \draw[->] (kk) to node {} (kinf);      
      \draw[-] (Ksp) to node {} (Kini);
      \draw[-] (Kini) to node {} (K);
      \draw[->] (K) to node {} (Kinf); 
      \draw[-] (Bk) to node {} (kk);
      \draw[-] (kk) to node {} (K);     
\end{tikzpicture}
\caption{Construction of the base-field $\K$} \label{fig:generalconstr}
\end{figure}
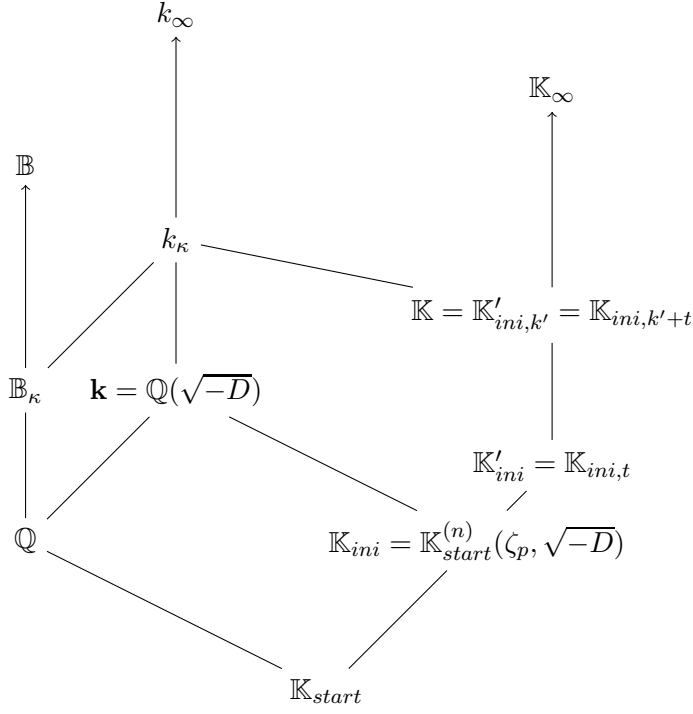

With this, we let $\K'_{ini} = \K_{ini,t} \subset \K_{ini,
  \infty}$. Finally we apply the Proposition \ref{tpdeco} in order to
choose a further extension of $\K'_{ini}$ in the same cyclotomic
$\Z_p$-extension, that yields some simple decomposition properties for
Thaine lifts.

Let $\kappa' = \kappa(\K'_{ini})$ and $\tilde{\tau} =
\gamma^{p^{\kappa'-1}}$ generate $\tilde{\Gamma} := \Gal\K'_{ini,
  \infty}/\K'_{ini}$.  Recall that $\gamma$ is a generator of
$\Gal(\B_{\infty}/\B)$, which explains the definition of
$\tilde{\tau}$; finally, $\tilde{T} = \tilde{\tau}-1$ and we let
$p^{b}$ be the exponent of $\id{M}(\apr{\K'_{ini}})$. We let $k'$ be
such that $\tilde{\omega}_{k'} \in ( p^{b+1}, T^{2(b+1)} )$ and define
$\K = \K'_{ini, k'} = \K_{ini,k+t}$ and let finally $\kappa =
\kappa(\K)$ and $\tau = \gamma^{p^{\kappa-1}}, T = \tau-1$, etc.  We
conclude from Proposition \ref{tpdeco} and the choice of $k'$ that
\begin{eqnarray}
\label{kdeco}
T \cdot \apr{\K} \subset \id{D}(\apr{\K}).
\end{eqnarray}
Moreover:
\begin{remark}
\label{kp1}
Suppose that $\KL/\K$ is a Thaine shift and $y
\in \apr{\KL}$ is such that either
\begin{itemize}
\item[ 1. ] $p y = x + w$ with $x \in \iota_{\KL/\K}(\apr{\K})$ and $w
  \in \id{M}(\apr{\KL})$, or
\item[ 2. ] $p^{b+1} y \in \id{L}(\apr{\KL})$. 
\end{itemize}
Then $T y \in \id{D}(\apr{\KL})$ too. 

The second point is a direct consequence of the choice of $k'$ and of
Proposition \ref{tpdeco}.  For the first, since $x \in \apr{\K}$, we
know that $p^b x \in \id{L}(\apr{\K})$, so $p^{b+1} y = p^b x - p^b w
\in \id{L}(\apr{\K}) + \id{M}(\apr{\KL}) \subset \id{D}(\apr{\KL})$,
and the fact follows from point 2.
\end{remark}

The construction of the Thaine shift is shown in the Figure
\ref{fig:thaine}
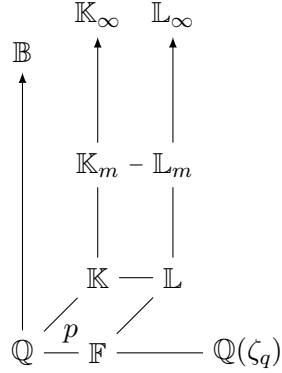
\begin{figure}
\centering
\begin{tikzpicture}[node distance = 1.5cm, auto]
      \node (Q) {$\mathbb{Q}$};
      \node (F) [right of=Q, node distance = 1cm] {$\mathbb{F}$};
      \node (Qz) [right of=F, node distance = 2cm] {$\mathbb{Q}(\zeta_q)$};
      \node (K) [above of=Q, right of=Q, node distance = 1cm] {$\mathbb{K}$};
      \node (L) [above of=F, right of=F, node distance = 1cm] {$\mathbb{L}$};
      \node (B) [above of=Q, node distance = 4cm] {$\mathbb{B}$};
      \node (Km) [above of=K] {$\mathbb{K}_m$};
      \node (Kinf) [above of=Km, node distance = 2cm] {$\mathbb{K}_{\infty}$};
      \node (Lm) [above of=L] {$\mathbb{L}_m$};
      \node (Linf) [above of=Lm, node distance = 2cm] {$\mathbb{L}_{\infty}$};
      \draw[-] (Q) to node {$\ \ p$} (F);
      \draw[-] (F) to node {} (Qz);
      \draw[-] (Q) to node {} (K);
      \draw[->,>=latex] (Q) to node {} (B);
      \draw[-] (F) to node {} (L);
      \draw[-] (K) to node {} (Km);
      \draw[-] (L) to node {} (Lm);      
      \draw[-] (K) to node {} (L);
      \draw[-] (Km) to node {} (Lm);      
      \draw[->,>=latex] (Km) to node {} (Kinf);
      \draw[->,>=latex] (Lm) to node {} (Linf);
\end{tikzpicture}
\caption{Thaine shift and lift} \label{fig:thaine}
\end{figure}

This concludes the sequence of steps for the construction of the base
field $\K$, which are reflected in the figure
\ref{fig:generalconstr}.  We review the conditions fulfilled by this
field:
\begin{itemize}
\item[ 0. ] The field $\K' = \K_{start}^{(n)}[ \zeta_p, \sqrt{-D} ]$
  and $\K_{ini} = \K'_s$ with $t$ subject to \rf{stab}.
\item[ 1. ] The field $\K$ is a Galois CM extension $\K/\Q$ which
  contains the \nth{p} roots of unity and such that $\mu(\K) > 0$ for
  the cyclotomic $\Z_p$-extension of $\K$. The primes that ramify
in $\K_{\infty}/\K$ are totally ramified.
\item[ 2. ] We have $T \cdot (\apr{\K}) \subset \id{D}(\apr{\K})$ and the
  properties in Remark \ref{kp1} are verified.
\item[ 3. ] The numbering of intermediate fields starts from $\kappa$,
  where $\K \cap \B = \B_{\kappa}$.
\item[ 4. ] The exponent $\exp(\id{M}(\apr{\K})) \geq p^2$.
\item[ 5. ] The field $\K$ contains an imaginary quadratic extension
  $\rg{k} = \Q[ \sqrt{ -d } ] \subset \K$ which has trivial $p$-part
  of the class group.  
  \end{itemize}

\subsection{Plan of the paper}
We choose a base field $\K$ as shown in the previous section and a
norm coherent sequence
\[ a = (a_n)_{n \in \N} \in \id{M}(\apr{\K}) \setminus \left( p \cdot
  \apr{\K} + (p, T) \id{M}(\apr{\K})\right). \] We note that condition
0. in the choice of $\K_{ini}$ readily implies that $\ord(a) =
\ord(a_1)$, so $(\ord(a)/p) \cdot a_1 \neq 0$ and thus
\begin{eqnarray}
 \label{abase}
 ( \ord(a)/p ) \in \id{M}(\apr{\K})[ p ] \setminus T \id{M}(\apr{\K})[ p ].
\end{eqnarray}
We let $o:=o_T(a) \geq 0 $ be such that $a \in (T^o) \cdot \apr{\K} \setminus
(T^{o+1}) \cdot \apr{\K}$.

In the second Chapter, we build a Thaine shift with respect to a prime
$\eu{q} \in a_m$ and a lift $b$ to $a$ and derive the main
cohomological properties of the shifted extension. The most important
facts are the decomposition $T x \in \id{D}(\apr{\KL})$ for all $x \in
\apr{\KL}$ with $\ell(x) \leq p \cdot \ord(\id{M}(\apr{\K})$ and
vanishing of the Tate cohomology $\wh{H}^0(F, \apr{\KL})$.  Based on
this and the fact that $T b$ is decomposed while $s b_m = 0$, as the
class of a ramified ideal, we obtain in Chapter 3 a sequence of
algebraic consequences which eventually lead to the fact that $T a =
\id{N}(b) \in \omega_m \id{M}(\apr{K})$; since this holds for
arbitrary choices of $m$, independently of $a$, we obtain a
contradiction to \rf{abase}, which proves the Iwasawa conjecture.

The paper is written so that the main ideas of the proof can be
presented in the main part of the text, leading in an efficient way to
the final proof.  The technical details and results are deduced with a
richness of detail, in the appendices.

\section{Thaine shift and proof of the Main Theorem}
We have selected in the first Chapter a base field $\K$ which is CM
and endowed with a list of properties. Consider the $\F_p[[ T
]]$-module $P := \apr{\K})/(p)$ and let $\pi : \apr{\K} \ra P$ be the
natural projection. Then for any $a \in \id{M}(\apr{\K}) \setminus p
\cdot \apr{\K}$ there is some integer $o_T(a) \geq 0$ such that the image
$\pi(a) \in \apr{\K}/(p)$ verifies $\pi(a) \in T^{o_T(a)}
\pi(\apr{\K})$. We choose $a \in \id{M}(\apr{\K}) \setminus p \cdot 
\apr{\K}$ with the minimal value of $o_T(a)$; let $m > \kappa(\K)$ be
such that
\begin{eqnarray}
\label{mdef}
\deg(\omega_m) > 2 (o_T(a) + 1)
\end{eqnarray}
and $\eu{q} \in a_m$ be a totally split prime. Let $q \subset \N$ be
the rational prime below $\eu{q}$; since $\K$ contains the \nth{p^m}
roots of unity, it follows that $q \equiv 1 \bmod p^m$. We let $\KL =
\K \cdot \F; \KL_{\infty} = \K_{\infty} \cdot \F$ be the Thaine shift
induced by $\eu{q}$, as described in the section \S 1.2 and let $b \in
\apr{\KL}$ be a Thaine lift of $a$.

For $C$ some $\Z_p[ s ]$-module, we use the Tate cohomologies
associated to $C$, defined by
\begin{eqnarray}
\label{tc}
\hat{H}^0(F, C) & = & \Ker(s : C \ra C)/(\id{N} C), \\
\hat{H}^1(F, C) & = & \Ker(\id{N} : C \ra C)/(s C). \nonumber
\end{eqnarray}
The notation introduced here will be kept throughout the paper.

Let $B'_n \subset A^-(\K_n)$ be the submodule spanned by the classes
of primes that ramify in $\KL_n/\K_n$. By choice of $\KL$, these are
the primes above $q$ and consequently $B'_n = \iota_{m,n}(B'_m)$ for
all $n > m$. Here $\iota_{m,n} : A^-(\KL_m) \ra A^-(\KL_n)$ is the
natural lift map. We let $p^v$ be the exponent of $B'_m$, so $p^v B'_n
= 0$ for all $n \geq m$.

Since $B'_n$ is constant up to isomorphism for all $n > m$, the
vanishing of $\hat{H}^0(F, \apr{\KL})$ is a straight forward 
consequence of   :

\begin{lemma}
\label{lh0}
\begin{eqnarray}
\label{h0} 
\Ker(s : \apr{\KL} \ra \apr{\KL}) = \iota(\apr{\K})
\end{eqnarray}
In particular, $\hat{H}^0(F, \apr{\KL}) = 0$.
\end{lemma}
\begin{proof}
  Consider $x = (x_n)_{n \in \N} \in \Ker( s : \apr{\KL} \ra
  \apr{\KL})$ and let $N > m + b$. Let $\eu{X} \in x_N$ be a prime:
  then $(\eu{X}^{s(1-\jmath)}) = (\xi^{1-\jmath})$, for some $\xi \in
  \KL_N$ and $\id{N}(\xi^{1-\jmath}) \in \mu(\K_N)$. Since $\eu{q}_m$
  is inert in $\K_N/\K_m$, we have $\id{N}(\KL_N) \cap \mu(\K_N)
  \subset \mu(\K_N)^p$. We may thus assume, after eventually modifying
  $\xi$ by a root of unity, that $\id{N}(\xi^{1-\jmath}) =
  1$. Hilbert's Theorem 90 implies that there is some $\alpha \in
  \KL_N$ such that
\[ \eu{X}^{s(1-\jmath)} = (\xi^{1-\jmath}) = (\alpha^{1-\jmath})^s \quad \Rightarrow \quad
 (\eu{X}/(\alpha))^{(1-\jmath)s} = (1).
\]

The ideal $\eu{Y} := (\eu{X}/(\alpha))^{1-\jmath} \in x_N^2$ verifies
$\eu{Y}^s = (1)$. If $x_N \not \in \iota(\apr{\K})$, then $\eu{Y}$
must be a product of ramified primes, so $x_N^2 \in B'_N +
\iota_{\K_N, \KL_N}(A_N^-(\K))$.  Recall that $B'_N =
\iota_{m,N}(B'_m)$ is spanned by the classes of the ramified primes
and $p^v B'_N = 0$. In particular $x_N^2 \in B'_N + \iota_{\K_N,
  \KL_N}(A_N^-(\K))$ and $B'_N = \iota_{m,N}(B'_m)$ imply that
\[ x_{N-v} = \Norm_{N, N-v} (x_N) \in \iota_{\K_{N-v}}, \KL_{N-v}(A^-(\K_{N-v})) + {B'_m}^{p^v} = \iota_{\K_{N-v}}, \KL_{N-v}(A^-(\K_{N-v})).\]  
This happens for all $N > m + v$, so $x \in \iota_{\K, \KL}(\apr{\K})$, as
claimed.
\end{proof}

Note that at finite levels we have 
\[ \hat{H}^0(F,A_n^-(\KL)) \cong B'_n/(B'_n \cap
\iota_{\KL/\K}(A_n^-(\K))) \neq 0. \] The Herbrand quotient of finite
groups is trivial, so $| \hat{H}^0(F,A_n^-(\KL)) | = | \hat{H}^1(F,A_n^-(\KL)) |$. 
In the projective limit however, $\wh{H}^1(F, \apr{\KL}) \neq 0$ so equality
is not maintained. In the Appendix \S 3.3, we prove though the following result:

\begin{lemma}
\label{hord}
Suppose that $v \in \id{M}(\apr{\KL})$ has non trivial image in the Tate group
$\wh{H}^1\left( F, \left( \id{M}(\apr{\KL}) \right) \right)$.
Then either $\ord(v) > p \exp(\id{M}(\apr{\K})$ or $T^2 v \in s \apr{\KL}$.
\end{lemma}

We turn now our attention to the decomposition of the Thaine lift $b$;
we prove in the Appendix 3.3 the following
\begin{proposition}
\label{px}
Let $F = < \nu >$ be a cyclic group of order $p$ acting on the
$p$-abelian group $B$ and let $x \in B$ be such that $y = \id{N}(x)$
has order $\ord(y) = q = p^l > p$.  Then $\ord(x) \leq p q$.
\end{proposition}

It implies
\begin{corollary}
\label{bdeco}
Let $y \in \apr{\KL}$ be such that $\ell(\id{N}(y)) = q > p$. Then $q
\leq \ell(y) \leq p q$ and $T y \in \id{D}(\apr{\KL})$.  All these
facts hold in particular for any Thaine lift $b$. In this case, one
has additionally $\ord( s b ) \leq \ord(b) \leq p q$.
\end{corollary}
\begin{proof}
  Let $f(T) \in \Z_p[ T ]$ be a distinguished polynomial that
  annihilates $p^{\ell(y)} y \in \id{L}(\apr{\KL})$ and let $\beta =
  f(T) y$. Then $\alpha := \id{N}(y)$ has the order $q > p$, by
  hypothesis, so we may apply the Proposition \ref{px}. This implies
  that $\ord(\beta) \leq p q$ and thus $\ell(y) \leq p q \leq p
  \exp(\id{M}(\apr{\K}))$, by definition of this order. The first claim
  in Remark \ref{kp1} implies that $T y \in \id{D}(\apr{\KL})$.  Since
  $\ord(a) = \ord(\id{N}(b)) > p$ by choice of $a$, the statement
  applies in particular to any Thaine lift $b$.  In this case, we know
  that $p b_m = a_m$ and $\ord(a_m) = q$, hence $\ord(b) \geq
  \ord(b_m) = p q$, hence $\ord(b) = p q$. We obviously have $\ord(s b) \leq \ord(b)$.
\end{proof}

\subsection{The vanishing of $\mu$}
Since $s b_m = 0$, the Theorem VI of Iwasawa \cite{Iw} implies that
there is some $c \in \apr{\KL}$ such that $s b = \nu_{m,1} c$. Then
$\nu_{m,1} (\id{N}(c) ) = 0$ and the Fact \ref{nonu} in the Appendix
implies that $\id{N}(c) = 0$. Moreover,
\[ p^{\ell(b)} c = p^{\ell(b)} (\nu_{m,1} s b ) \in
\id{L}(\apr{\KL}) \] so, by Corollary \ref{bdeco}, $\ell(c) \leq p q$
too, and thus $T c \in \id{D}(\apr{\KL})$. Let $T c = c_{\lambda} +
c_{\mu}$. Since $\id{L}(\apr{\K}) \cap \id{M}(\apr{\K}) = 0$, it
follows that $\id{N}(c_{\lambda}) = \id{N}(c_{\mu}) = 0$,
individually. We have, by comparing parts, $s b_{\mu} = \nu_{m,1}
c_{\mu}$, so $p q \cdot \nu_{m,1} c_{\mu} = 0$, and since
$\id{M}(\apr{\KL})$ is rigid, \rf{rigid} implies that $\ord(c_{\mu})
\leq p q$. We may thus apply Lemma \ref{hord}, which implies that $T^2
c_{\mu} \in s \id{M}(\apr{\KL})$, say $T^2 c_{\mu} = s x$; then $s
(T^2 b_{\mu} - \nu_{m,1} x) = 0$ and Lemma \ref{lh0} implies that
$T^2 b_{\mu} = \nu_{m,1} x + z, z \in \iota(\apr{\KL})$. By taking
norms we obtain $T^3 a = \nu_{m,1} x + p z$. This implies $o_T(a)+3
\geq \deg(\nu_{m,1})$ and we can choose $m$ large enough, to obtain a
contradiction. This confirms the claim of the Main Theorem.

\section{Appendix}
In the Appendix, unless otherwise specified, the notation used in the
various Facts and Lemmata is the one used in the section where these
are invoked in the text.  In the next section we provide a list of
disparate, useful facts:
\subsection{Auxiliary facts}
We start by proving that if $\mu >0$ for some number field, then it
is also non - trivial for finite extensions thereof.   
\begin{fact}
\label{blow}
Let $K$ be a number field for which $\mu(K) \neq 0$ and $L/K$ be
a finite extension, which is Galois over $K$. Then $\mu(L) \neq
0$.
\end{fact}
\begin{proof}
  If $M \subset L$ has degree coprime to $p$, then $\Ker( \iota : A(K)
  \ra A(M)) = 0$, so we may reduce the proof to the case of a cyclic
  Kummer extension of degree $p$. Let $M = L^{\Gals(L/K)_p}$ be the
  fixed field of some $p$-Sylow subgroup of $\Gal(L/K)$. Then $p$ does
  not divide $[ M : K ]$, so $\Ker(\iota : A^-(K) \ra A^-(M)) = 0$, and thus $\mu(M) \neq 0$. We may assume
  without loss of generality, that $M$ contains the \nth{p} roots of
  unity. Since $p$-Sylow groups are solvable, the extension $L/M$
  arises as a sequence of cyclic Kummer extensions of degree $p$. It
  will thus suffice to consider the case in which $k$ is a number
  field with $\mu \neq 0$ and containing the \nth{p} roots of unity
  and $k' = k[ a^{1/p} ]$ is a cyclic Kummer extension of degree $p$.
  We claim that under these premises, $\mu(k') \neq 0$.  Let $k_n
  \subset k_{\infty}$ and $k'_n \subset k'_{\infty}$ be the
  intermediate fields of the cyclotomic $\Z_p$-extensions, let $\nu$
  generate $\Gal(k'/k)$.  Let $F/k_{\infty}$ be an abelian unramified
  extension with $\Gal(F/k_{\infty}) \cong \F_p[[ T ]]$; such an
  extension must exist, as a consequence of $\mu > 0$. There is thus
  for each $n > 0$ a $\delta_n \in k_{n}^{\times}$ such that $F_n =
  k_{n}\left[ \delta_n^{\F_p[[ T ]]/p} \right]$ is an unramified
  extension with galois group $G_n = \Gal(F_n/k_n)$ of $p$-rank $r_n
  := \prk(G_n) > p^{n-c}$ for some $c \geq 0$. We define $F'_n = F_n[
  a^{1/p} ]$ and let $\overline{F}'_n \supset F'_n$ be the maximal
  subextension which is unramified over $k'_n$. We have
  $\overline{F}'_n \supseteq F_n$ and thus $\prk(\Gal( \overline{F}'_n
  / k'_n ) ) \geq \prk(\Gal(F_n/k_n)) \ra \infty$.  Consequently,
  $k'_{\infty}$ has an unramified elementary $p$-abelian extension of
  infinite rank, and thus $\mu(k') > 0$, which completes the proof.
\end{proof}

\begin{fact}
\label{b2}
Let $\K$ be a CM extension with $\mu > 0$. Then it is possible to
build a further CM extension $\KL/\K$ with $\exp(\id{M}^-(\KL)) >
p^2$.
\end{fact}
\begin{proof}
  We have shown in Fact \ref{blow} that we may assume that $\mu_{p^3}
  \subset \K$. Let $a = (a_n)_{n \in \N} \in \id{M}^-(\K)$ and $\eu{q}
  \in a_2$, with $a_2 \neq 0$, be a totally split prime which is inert
  in $\K_{\infty}/\K_2$. Let $\KL/\K_2$ be the inert Thaine shift of
  degree $p^2$ induced by $\eu{q}$, let $b_2 = [ \eu{Q}^{(1-\jmath)/2}
  ]$ be the class of the ramified prime of $\KL$ above $\eu{q}$ and $b
  = (b_m)_{m \in \N}$ be a sequence through that extends $b_2$, such that
  $N_{\KL/\K}(b) = a$.  Then $b \not \in \id{L}^-(\KL)$ and there is
  some polynomial $f(T) \in \Z_p[ T ]$ such that $f(T) b \in
  \id{M}^-(\KL)$, while $\Norm_{\KL/\K}(f(T) b) = f(T) a$.  The
  capitulation kernel $\Ker(\iota : A^-(\K_n) \ra A^-(\KL))$ is
  trivial and consequently $\ord(T f(T) b) \geq p^2 \ord(a)$; hnece
  $\exp(M^-(\KL)) > p^2$. Thus $\KL$ verifies the claimed properties.
\end{proof}
The following fact was proved by Sands in \cite{Sd}:
\begin{fact}
 \label{sa}
 Let $\KL/\K$ be a $\Z_p$-extension of number fields in which all the
 primes above $p$ are completely ramified. If $F(T) \in \Z_p[ T ]$ is
 the minimal annihilator polynomial of $\id{L}(\KL)$, then $(F, \nu_{n,1})
 = 1$ for all $n > 1$.
\end{fact}
As a consequence,
\begin{corollary}
\label{nonu}
Let $\K$ be a CM field and suppose that $x \in \apr{\K}$ verifies
$\nu_{n,1} x = 0$.  Then $x = 0$.
\end{corollary}
\begin{proof}
  Let $q$ be the exponent of the $\Z_p$-torsion of $\apr{\K}$.  It
  follows then that $q x \in \id{L}(\apr{\K})$ is annihilated by
  $\nu_{n,1}$, so the Fact \ref{sa} implies that $q x = 0$ and thus $x
  \in \id{M}(\apr{\K})$. Since $\nu_{m,1} x = 0$ it follows that $x =
  0$, as claimed.
\end{proof}


\subsection{Applications of the Tchebotarew Theorem}
We prove the existence of Thaine lifts.
\begin{lemma}
\label{thlift}
Let $\K$ be a CM field and $a = (a_n)_{n \in \N} \in \apr{\K}$ and
$\KL = \K \cdot \F$ be a Thaine shift induced by a split prime $\eu{q}
\in a_m$.  Then there is a Thaine lift $b = (b_n)_{n \in \N} \in
\apr{\KL}$ with the properties that $\id{N}(b) = a$ and $b_m$ is the
class of the ramified prime $\eu{Q} \subset \KL$ above $\eu{q}$.
\end{lemma}
\begin{proof}
  We prove by induction that for each $n \geq 1$ there is a class $b_n
  \in A_n^-(\KL)$ with $\id{N}(b_n) = a_n$ and $N_{n,n-1}(b_n) =
  b_{n-1}$. The claim holds for $n \leq m$ by definition. Assume that
  it holds for $n \geq m$ and consider minus parts of the maximal
  $p$-abelian unramified extensions 
  \[ \KH(\K_n)^-/ \K_n, \KH^-(\KL_n)/\KL_n, \KH^-(\KL_{n+1})/\KL_{n+1}, \quad \hbox{ and } \quad \KH^-(\K_{n+1})/\K_{n+1}. \]
  For $\KH' \in \{ \KH^-(\K_n), \KH^-(\KL_n), \KH^-(\K_{n+1}) \}$, we obviously have $\KH' \cdot
  \KL_{n+1} \subset \KH^-(\KL_{n+1})$.  For some unramified
  Kummer extension $M/K$ we denote by $\varphi(x) = \lchooses{M/K}{x}$
  the Artin symbol of a class or of an ideal.  The induction
  hypothesis implies that both $\varphi(b_n)$ and $\varphi(a_{n+1})$
  restrict to $\varphi(a_n) \in \Gal(\KH^-(\K_n)/\K_n)$, and since
  $A^-_{n+1}(\K)$ and $A^-_n(\KL)$ surject by the respective norms to
  $A^-_n(\K_n)$, it follows that
  \[ \KL_{n+1} \KH(\K_{n+1}) \cap \KL_{n+1} \KH^-(\KL_n) = \KL_{n+1}
  \KH^-(\K_n). \] There is thus some automorphism $x \in
  \Gal(\KH^-(\KL_{n+1})/\KL_{n+1})$, such that
\[ x \big \vert_{\Gals(\KH^-(\KL_n)/\KL_n)} = \varphi(b_n) \quad \hbox{ and } \quad
 x \big \vert_{\Gals(\KH^-(\K_{n+1})/\K_{n+1})} = \varphi(a_{n+1}).
\]
By Tchebotarew, there are infinitely many totally split primes $\eu{R}
\subset \KL_{n+1}$ with Artin symbol
$\lchooses{\KH^-(\KL_{n+1})/\KL_{n+1}}{\eu{R}} = x$, and by letting
$b_{n+1} = [ \eu{R} ]$ for such a prime, we have $\id{N}(b_{n+1}) =
a_{n+1}$ and $N_{n+1,n}(b_{n+1}) = b_n$. We obtain by induction a norm
coherent sequence $b = (b_n)_{n \in \N}$ which verifies $b_m = [
\eu{Q} ]$ and $\id{N}(b) = a$, and this completes the proof.
\end{proof}

\subsection{Proof of Proposition \ref{px} and related results}
Let the notations be like in the statement of the Proposition and let
$q = p^k = \ord(y)$ and $r \geq q$ be the order of $x$ and $\rg{R} =
\ZM{r}[ s ]$. Then $\rg{R}$ has maximal ideal $(p, s)$ and $s$ is
nilpotent. By definition, $\rg{R}$ acts on $x$ and we consider the
modules $X = \rg{R} x$ and $Y = \ZM{q'} y \subset X$.

With these notations, we are going to prove that $p q x = 0$. We start
by proving
\begin{lemma}
\label{h0fin}
Under the given assumptions on $x, y$ and with the notations above,
\[ \hat{H}^0(F, X) = 0, \] or $q p x = 0$. Moreover, $s X \cap Y \cong
\F_p \subset X[ s, p ]$ in all cases, while if $\hat{H}^0(F, X) \neq
0$, the bi-torsion $X[ s, p ] \cong \F_p^2$ and $q' x \in X[ s, p ] \setminus s X$.
\end{lemma}
\begin{proof}
  Let $N$ be a sufficiently large integer,
  such that $p^N x = 0$ and let $R = \ZM{p^N}$, so $X$ is an 
  $R[ s ]$-module. We note that in $R' := R/(\id{N})$ the maximal ideal is
  $(s)$, thus a principal nilpotent ideal, as one can verify from the
  definition of the norm $\id{N} = \sum_{i=0}^{p-1} \nu^i$. Indeed, in
  $R$ we have the identity $p = s^{p-1} \cdot v(s) \bmod \id{N}, v(s)
  \in R^{\times}$, so the image of $p$ in $R'$ is a power of $s$,
  hence the claim. As a consequence, in any finite cyclic $R'$-module $M$
  we have $M[ p, s ] \cong \F_p$.
 
  We consider the two modules $N_1 = \Z_p x$ and $N_2 = R s x$, such that
  $X = N_1 + N_2$, and only $N_2$ is an $R$-module, and in fact an $R'$-
  module, since it is annihilated by $\id{N}$\footnote{This proof is
    partially inspired by some results in the Ph. D. Thesis of Tobias
    Bembom \cite{Be}}. Note that $q N_1$ is also annihilated by the norm;
    this covers also the case when $q x = 0$, so $q N_1 = 0$.
  Since $X = N_1 + N_2$, it follows that 
  $X[ p ] = N_1[ p ] + N_2[ p ]$.
  
  We always have that $N_2[ s ] \cong \F_p$, since $N_2$ is an
  $R'$-module and $(s)$ is the maximal ideal of $R'$. Let $K := X[ s ]
  = \Ker( s : X \ra X )$ and assume that $H_0 := \hat{H}^0(F, X) \neq
  0$; then the bi-torsion $X[p, s] = X[ s ][ p ]$ is not
  $\F_p$-cyclic. It follows that $X[ p, s] \not \cong N_1 \cap X[ p , s ] \cong \F_p$.
  Let $r = q p ^e = \ord(x)$, with $e \geq 0$. Note that
  \[ y_0 := (q/p) y = q x u(s) + (q/p) s^{p-1} \in X[ p, s ], \] for
  all values of $e$. If $e = 0$, then $y_0 = (q/p) s^{p-1} \in N_2[ p,
  s ]^{\times}$. But then $(q/p) s x \neq 0$ and thus $N_1[ p, s ] =
  0$ so $\wh{H}(F, X) = 0$ in this case.  Assume that $e > 0$ and $y_0
  \in N_2$. Since $y_0 = q x + (q/p) f(s) ( s x )$ and the last term,
  $y_1:=(q/p) f(s) ( s x ) \in N_2$, it follows that $q x = y_0-y_1
  \in N_2$, so $\wh{H}(F, X) = 0$ in this case too. It remains that if
  $\wh{H}(F, X) \neq 0$, then $y_0 \in N_1 \setminus N_2$ and $e > 0$. 
  
  In this case $N_1 \cap N_2 = \emptyset$ and $X = N_1 \oplus N_2$. We let $(q/p)
  (p^e x - c y) = 0$, with $(c, p) = 1$. Then there is some $w \in X[
  q/p ]$ with $p^e x = c y + w$, and the decomposition $X = N_1 \oplus
  N_2$ induces a decomposition of the $q/p$-torsions, so we may write
  $w = a (q/p) + b v(s) s^N x$, with $N > 0$ and $a, b \in \{ 0, 1,
  \ldots, p-1\}$, $v(s) \in {R'}^{\times}$.  Taking norms in the
  identity $(p^e - a (q/p) ) x = c y + b v(s) s^N x$ we obtain
  \[ p y \cdot ( c - p^{e-1} + a (q/p^2) ) = 0.\] We have assumed $q >
  p^2$, so for $e > 1$ the cofactor of $p y$ is a unit, which implies $p
  y = 0$, in contradiction with $\ord(y) = q > p^2$. Therefore $e = 1$
  and
  \begin{eqnarray}
  \label{hneq1}
  y = p x + v(s) s^{p-1} x, \quad \hbox{ and } \quad (q/p) s^{p-1} x = 0. 
  \end{eqnarray}
  
  Consequently, if $\wh{H}(F, X) \neq 0$, then $\ord(x) = q p$ and
  there is an $N < (p-1) k$ such that $N_2[ s ] = \F_p \cdot (s^N x)$.
  This completes the proof.
  
\end{proof} 
 
We can now assume that $\hat{H}^0(F, X) = 0$ and since $X$ is finite, the
Herbrand quotient vanishes and $\hat{H}^1(F, X) = 0$. 
There is thus an exact sequence of $\Z_p[ s ]$-modules 
\[ 0 \ra X[ p ] \ra X \ra X \ra X/p X \ra 0, 
\]
in which $X/ p X$ is cyclic generated by $x$ and the arrow $X \ra X$
is the multiplication by $p$ map $ \cdot p$. 

Since $ \hat{H}^0(F, X) =
0$, it follows that $\Ker(s : X \ra X)[ p ] \cong \F_p$ and $X[ p ]$
is $\F_p[ s ]$-cyclic, as is $X/p X$. Using this observation, we
provide now the proof of the Proposition.
\begin{proof}
  There is some distinguished polynomial $\phi \in \rg{R}$ with
  $\phi(s) X = 0$ and $\deg(\phi) = \prk(X)$. Indeed, let $d =
  \prk(X)$ and $\overline{x} \in X/p X$ be the image of $x$, so $(s^i
  b)_{i=0,d-1}$ have independent images in $X/p X$, by definition of
  the rank and span $X$, as a consequence of the Nakayama
  Lemma. Therefore $s^d x \in p X$ and there is a monic distinguished
  annihilator polynomial
  \[ \phi(s) = s^d - p^e h(s) \] of $X$, with $e \geq 0$ and $h$ a
  polynomial of $\deg(h) < \deg(\phi) \leq p$, which is not
  $p$-divisible. Note that, by minimality of $d$, the case $e = 0$
  only occurs if $\phi(s) = s^d$. We shall distinguishe several cases
  depending on the degree $d = \deg(\phi)$ and on properties of
  $h(s)$.

  We consider the cases in which $d < p$ first. Assume that $s^c x =
  0$ for some $c < p$, so $h = 0$. Upon multiplication with
  $s^{p-1-c}$ we obtain $s^{p-1} X = 0$. Thus $\id{N} ( x ) = y =
  (s^{p-1} + p u(s)) x = p u(s) x$, and $y u^{-1}(s) = y = p x$, which
  settles this case.

  We can assume now that $\phi(s) = s^d - p^e h(s)$ and $e > 0$, with $h(s)$
  a non trivial distinguished polynomial of degree $\deg(h) < d < p$. As a consequence, if $d < p-1$, then
\begin{eqnarray*}
y & = & \id{N}(x) = p u(s) x + s^{p-1} x = p (u(s) + s^{p-1-d} p^{e-1}
h(s)) x \\ & = & p v(s) x, \quad v(s) = u(s) + s^{p-1-d} p^{e-1} h(s)\in \rg{R}^{\times}.
\end{eqnarray*}
Hence $y = v^{-1}(s) z = p x$, which confirms the claim in this case.

If $d = p-1$, then $s^{p-1} x =  p \cdot (p^{e-1} h(s)) x$ and thus  
\[ y = (p u(s) + s^{p-1}) x = p(u(s) + p^{e-1} h(s)) x;
\]
if the expression in the brackets is a unit, we may conclude like
before. Otherwise, $e=1$ and $h(s) = -1 + s h_1(s)$, and thus $y \in s
X$, so $\id{N}(y) = p y = 0$, in contradiction with the
assumption\footnote{At this point the assumption that $\ord(y) > p$
  plays a crucial role, and if it were not to hold, modules such that
  $\ord(x)$ becomes arbitrarily large are conceivable} that $\ord(y) >
p$.

The case $d = p$ is more involved. We claim that $\ord(x) < p^2
\ord(y)$. Assume that this is not the case. Since $d = p$, we have
$\prk(X) = p$ and thus $s^{p-1} x = y - p x u(s) \not \in p X$. In
particular $y \not \in p X$. Let $q = p^{k} = \ord(y)$ and assume that
$\ord(x) = p^{e+k} = p^e \cdot \ord(y), e > 1$. We note that
$\ord(s^{p-1} x) = p^{e-1} q$, since $s^{p-1} u^{-1}(s) x = - p x + y$
has annihilator $p^{e-1} q$. Consider the generators $s^j x$ of $X$; 
there is an integer $j$ in the interval $0 \leq j < p-1$, such that
\[ \ord(s^j x) = \ord(x) > \ord(s^{j+1} x) = \ord(s^{p-1} x) =
\ord(x)/p. \] Recall from Lemma \ref{h0fin}, that we are in the case when
$X[ p ]$ is a cyclic $\F_p[ s ]$ module of dimension $p$ as an $\F_p$-vector
space, and $\wh{H}^0(F, X) = 0$. Let
\begin{eqnarray*}
  \id{F}_0 & := & \{ q p^{e-1} s^i x \ : \ i = 0, 1, \ldots, j \}  \subset X[ p ],  \\
  \id{F}_1 & :=& \{ q p^{e-2} s^{j+i} x \ : \ i = 1, 2, \ldots, p-j-1 \}  \subset X[ p ] , 
\end{eqnarray*}
and $\id{F} = \id{F}_0 \cup \id{F}_1$. Then $\id{F}_i \subset X[ p ]$
are $\F_p$-bases of some cyclic $\F_p[ s ]$ submodules $F_0, F_1
\subset X[ p ]$ with $\dim_{\F_p}(F_0) \leq j+1$ and $\dim_{\F_p}(F_0)
\leq p-(j+1)$. We claim that $X[ p ] = F_0 \oplus F_1 = 0$.  For each
$z \in X[ p ]$ there is some maximal $z' \in X$ -- thus $z'$ having
non-trivial image $0 \neq \overline{z}' \in X/p X$ -- and such that $z
= q' z'$ for some $q' \in p^{\N}$. Since the generators of $X/p X$ are
mapped this way to $F = F_0 + F_1$, which is an $\F_p$-vector space,
it follows by linearity that $F = X[ p ]$.  Comparing dimensions, we
find that $F_0 \cap F_1 = 0$ so there is a direct sum $F = F_0 \oplus
F_1 = X[ p ]$, as claimed.

Note that
\[ 0 \neq (q/p) y = q x + (q/p) s^{p-1} u^{-1}(s) x \in X[ p ][ s ] ; \] 
upon multiplication with $p^{e-1} \geq p$ we obtain 
\begin{eqnarray}
\label{s1}
0  & = &  q p^{e-2} y = q p^{e-1} u(s) x + q p^{e-2} s^{p-1} x \nonumber \\
& = &  q p^{e-1} (u_0(s) + u_1(s) ) x + q p^{e-2} s^{p-1} x,  
\end{eqnarray}
where $u_0 \in F_0$ and $u_1 \in F_1$ are the projections of the unit
$u(s)$, thus
\[ u_0(s) \equiv \frac{\sum_{i=0}^{j} s^i \binom{p}{i}}{p} \bmod \ p,
\quad u_1(s) \equiv \frac{\sum_{i=j+1}^{p-2} s^i \binom{p}{i}}{p}
\bmod \ p . \] By definition of $j$, it follows that
  \[ f_0 := q p^{e-1} u_0(s) x \in F_0, \quad \ f_1 := s^{p-1} q
  p^{e-2} x \in F_1, \] and $q p^{e-1} u_1(s) x = 0$. The identity in
  \rf{s1} becomes $f_0 + f_1 = 0$, and since $F_0 \cap F_1 = 0$ and
  $f_i \in F_i, \ i = 0, 1$, it follows that $f_0 = f_1 = 0$. However,
  $f_1 = s^{p-1} q p^{e-2} x \in \id{F}_1$ is a basis element which generates
  $F_1[ s ]$, so it cannot vanish. The contradiction
  obtained implies that we must have in this case also $\ord(x) \leq p q$.
 Consequently, $\ord(x) \leq p \cdot \ord(y)$ in all cases, which
  completes the proof of the Proposition.
\end{proof}

\subsection{The Norm Principle, ray class fields and proof of Lemma \ref{hord}}
Let $q$ be a rational prime with $q \equiv 1 \bmod p^m$. Let $\rg{F}
\subset \Q_q[ \zeta_q ]$ be the subfield of the (ramified) \nth{q}
cyclotomic extension, which has degree $p$ over $\Q_q$. Thus $\rg{F}$
is the completion at the unique ramified prime above $q$ of the field
$\F$ defined in the text. The field $\rg{F}$ is a tamely ramified
extension of $\Q_q$, so class field theory implies that
$\Gal(\rg{F}/\Q_q)$ is isomorphic to a quotient of order $p$ of $\ZMs{q}$, 
so letting $S = (\ZMs{q})^p$, we have $\Gal(\rg{F}/\Q_q) \cong \ZMs{q}/S$. 
Let thus $r \in \Z$ be such that 
$r^{(q-1)/p} \not \equiv 1 \bmod q$: if $g \in \F_q^{\times}$
generates the multiplicative group of the finite field with $q$
elements and $m = v_p(q-1)$, then one can set $r = g^{(q-1)/p^m} \ \ \rem
\ \ q$. Let in addition $\rg{K}_n/\Q_q$ be the \nth{p^n} cyclotomic
extension of $\Q_q$: under the given premises, we have for $n > m$ the
extension degree $[ \rg{K}_n : \Q_q ] = p^{n-m}$ and $\rg{K}_n =
\Q_q\left[ r^{1/p^{n-m}} \right]$, while for $n \leq m$ the extension
is trivial.  Letting $r_n = r^{1/p^{n-m}} \in \rg{K}_n$, and $\rg{L}_n
= \rg{K}_n \cdot \rg{F}$, we deduce by class field theory that $r_n$
generates $\rg{K}_n^{\times}/\id{N}(\rg{L}_n^{\times})$, under the
natural projection. Indeed, the extension $\rg{L}_n/\rg{K}_n$ is a
ramified $p$ extension, so the galois group must be a quotient of the
roots of unity $W(\Z_q)$, hence the claim.  We shall use these
elementary observation in order to derive the structure of
$\wh{H}^1(F, A^-(\KL_n))$ in our usual setting and prove some
necessary conditions for elements $v_n \in A^-(\KL_n)$ which verify $0
\neq \beta(v_n) \in \wh{H}^1(F, A^-(\KL_n))$, under the natural
projection $\beta : A^-(\KL_n) \ra \wh{H}^1(F, A^-(\KL_n))$.

Let $\K, \KL, \F$, etc. be the fields defined in the main part of
the paper and let us denote by $I(\M)$ the ideals of some arbitrary
number field, and $P(\M) \subset I(\M)$ the principal ideals. The
maximal $p$-abelian unramified extension is $\KH(\M)$ and the $p$-part
of the ray class field to the ray $\eu{M}_q = q \cdot \id{O}(\M)$ will
be denoted by $\T_q(\M)$. If $\M$ is a CM field, then complex
conjugation acts, inducing $I^-(\M), P^-(\M)$ and $\KH^-, \T_q^-$, in
the natural way. In our context, we let in addition $P_N^-(\KL) :=
\id{N}(P^-(\KL)) \subset P^-(\K)$.  We let $\Delta_n =
\Gal(\K_n/\rg{k})$ for arbitrary $n$.  The following is an elementary
result in the proof of the Hasse Norm Principle:
\begin{lemma}
\label{princid}
Let $\K, \KL$ and $F$ be like above. Then
\begin{eqnarray}
\label{h1gen}
\wh{H}^{(1)}( F, A^-(\KL_n) ) \cong P^-(\K_n)/P_N^-(\KL_n) \cong \F_p[ \Delta_m ], \quad \hbox{for all $n > 0$},
\end{eqnarray}
the isomorphism being one of cyclic $\F_p[ \Delta_m ]$-modules.
\end{lemma}
\begin{proof}
  Note that both modules in \rf{h1gen} are annihilated by $p$. In the
  case of $P^-(\K)/P_N^-(\KL)$, this is a direct consequence of
  $(P^-(\K))^p = \Norm_{\KL/\K}(P^-(\K)) \subset P_N^-(\KL)$. Let
  $\beta : A^-(\KL) \ra \wh{H}^{(1)}( F, A^-(\KL) )$ and $\pi_N :
  P^-(\K) \ra P^-(\K)/P_N^-(\KL)$ denote the natural projections and
  let $a \in \Ker( \id{N} : A^-(\KL) \ra A^-(\KL) )$. Then $p u(s) a =
  - s^{p-1} a$ and thus $p a = - s^{p-1} u^{-1}(s) a$ and a fortiori
  $\beta(p a) = 0$ for all $a$, so $p \wh{H}^{(1)}( F, A^-(\KL) ) =
  0$, thus confirming that $\wh{H}^{(1)}( F, A^-(\KL_n) )$ is an $\F_p$-module too.

  Let now $\eu{A} \in a$ be some ideal and $(\alpha) =
  \id{N}(\eu{A})$. The principal ideal $\eu{a} :=
  (\alpha/\overline{\alpha}) \in P^-(\K)$ has image $\pi_N(\eu{a}) \in
  P^-(\K)/P_N^-(\KL)$ which depends on $a$ but not on the choice of
  $\eu{A} \in a$. This is easily seen by choosing a different ideal
  $\eu{B} = (x) \eu{A} \in a$: then $\id{N}(\eu{B}^{1-\jmath}) =
  \eu{a} \cdot \id{N}(x/\overline{x}) \in \eu{a} \cdot P_N^-(\KL)$,
  and $\pi_N(\id{N}(\eu{B}^{1-\jmath})) =
  \pi_N(\id{N}(\eu{A}^{1-\jmath})) = \pi_N(\eu{a})$ depends only on
  $a$.  Suppose now that $\eu{a} \in P^-_N(\KL)$, so $\pi_N(\eu{a}) =
  1$. Then there is some $y \in \KL^{\times}$ such that
  $\id{N}(\eu{A}^{1-\jmath}) = (\id{N}(y)^{1-\jmath})$ and thus
  $\id{N}(\eu{A}/(y))^{1-\jmath} = (1)$. Since $\wh{H}^{(1)}$ vanishes
  for ideals, it follows that there is a further ideal $\eu{X} \subset
  \KL$ such that
  \[ \eu{A}^{1-\jmath} = \left((y) \eu{X}^s \right)^{1-\jmath} , \]
  and thus $a^2 \in (A^-(\KL))^s$. But then $\beta(a) = 0$. We have
  shown that there is a map $\lambda : \wh{H}^{(1)}( F, A^-(\KL) ) \ra
  P^-(\K)/P_N^-(\KL)$ defined by the sequence of associations
  $\beta(a) \mapsto \eu{a} \mapsto \pi_N(\eu{a})$, which is a well
  defined injective homomorphism of $\F_p$-modules.

  In order to show that $\lambda$ is an isomorphism, let $\eu{x} :=
  (x/\overline{x}) \in P^-(\K) \setminus P_N^-(\KL)$ be a principal
  ideal that is not a norm from $\KL$. Let the Artin symbol of $x$ be
  $\sigma = \lchooses{\KL/\K}{x} \in \Gal(\KL/\K)$; by definition,
  $\KL$ is also $CM$, so complex conjugation commutes with $\sigma$
  and we have
\[ \lchooses{\KL/\K}{\overline{x}} = \lchooses{\KL/\K}{x^{\jmath} } =
\sigma^{\jmath} = \sigma . \] Consequently,
$\lchooses{\KL/\K}{(x/\overline{x})} = 1$ -- we may thus choose, by
Tchebotarew, a principal prime $(\rho) \subset \K$ with $\rho \cong
x/\overline{x} \bmod q$, and which is split in $\KL/\K$. Let $\eu{R}
\subset \KL$ be a prime above $(\rho)$ and $r := [ \eu{R}^{1-\jmath} ]
\in (A^-)(\KL)$.  We claim that $\beta(r) \neq 0$; assume not, so
$\eu{R} = (y) \eu{Y}^s$ for some $y \in \KL$ and $\eu{Y} \subset \KL$
and thus $\id{N}(\eu{R}^{1-\jmath}) = (\id{N}(y/\overline{y})) \cong
\eu{x} \bmod P^-_N(\KL)$. Since $(\id{N}(y/\overline{y})) \in
P^-_N(\KL)$ by definition, it follows that $\eu{x} = (x/\overline{x})
\in P_N^-(\KL)$, which contradicts the choice of $x$. It follows that
$\lambda$ is an isomorphism of $\F_p[ \Delta_m ]$-modules.  The proof
will be completed if we show that $P^-(\K)/P_N^-(\KL) \cong \F_p[ \Delta_m ]$.
Since $\Delta_m$ acts transitively on the pairs of complex conjugate primes above $q$
in $\K_n$, one verifies that $(\K_n^{\times})^-/\id{N}((\KL_n^{\times})^-) \cong \F_p[ \Delta_m ]$.
The field $\rg{k}$ contains no \nth{p} roots of unity, and thus
$\Norm_{\K_n/\rg{k}}( P^-(\K_n)/P_N^-(\KL_n) ) \neq 0$. The isomorphism
\[ P^-(\K)/P_N^-(\KL) \cong \F_p[ \Delta_m ] \]
follows, and this completes the proof.
\end{proof}

Consider the field $\T'_n := \T^-_q(\K_n)$ defined as the minus
$p$-part of the ray class field to the modulus $Q_n := q \id{O}(\K_n)$
-- i.e., if $\rg{T}$ is the full ray class field to $Q_n$ and
$\rg{T}_p$ is the $p$-part of this field, thus fixed by all $q'$-Sylow
subgroups of $\Gal(\rg{T}/\K_n)$, for all primes $q' \leq q$, then
$\T'_{n}$ is the subfield of $\rg{T}_p$ fixed by
$(\Gal(\rg{T}_p/\K_n))^+$. As mentioned above, the completions at
individual primes $\eu{q}'$ above $q$ are cyclic groups
\[ \T'_{n,\eu{q}'}/\KH(\K_n)_{\eu{q}'} \cong (\F_q[ \zeta_{p^n} ]^{\times})_p \cong C_{p^n} .\] 
Consequently, 
\begin{eqnarray}
\label{tn}
\Gal(\T'_n/\KH(\K_n)) \cong \prod_{g \in \Delta_m} C_{p^n} .
\end{eqnarray}
Since $q$ is ramified in $\KL_n/\K_n$, the residual fields of the ray
class subfield $\T_n := \T^-_q(\KL_n)$ are the same as the ones of
$\T'_n$ and thus $\T_n = \KH(\KL_n) \cdot \T'_n$.  Let also $\E'_n
\subset \T'_n$ be the maximal subextension with $p \Gal(\E'_n/\KH'_n)
= 0$ -- thus the $p$-elementary extension of $\KH(\K_n)$ contained in
$\T'_n$ -- and $\E_n = \E'_n \cdot \KH(\KL_n)$. Since the local
extensions $\T'_{n,\eu{q}'}/\KH(\K_n)_{\eu{q}'}$ are cyclotomic, the
ramification of $\E'_n/\KH(\K_n)$ is absorbed by $\KL_n/\K_n$ and
therefore $\E_n \subset \KH(\KL_n)$.

Consider a class $x \in A^-(\KL_n)$ such that $0 \neq \beta(x) \in
\wh{H}^1(F,A^-(\KL_n))$ and let $\eu{A} \in x$ be a prime and
$(\alpha) = \id{N}(\eu{A}) \subset \K_n$ be the principal ideal below
it. By Lemma \ref{princid}, it follows that
$\pi_N((\alpha/\overline{\alpha})) \neq 0$ and thus the Artin symbol
$y' = \lchooses{\T'_n/\K_n}{(\alpha/\overline{\alpha})} \in
\Gal(\T'_n/\K_n)$ generates a cycle of maximal length in
$\Gal(\T'_n/\KH(\K_n))$, so it acts non trivially in
$\E'_n/\KH(\K_n)$. If $y \in \Gal(\T_n/\KL_n)$ is any lift of
$\varphi(x)$, i.e. $y \vert_{\KH(\KL_n)} = \varphi(x)$, then
$\id{N}(y')$ acts non trivially in $\E'_n/\KH(\K_n)$.  The converse
holds too. Suppose that $x \in A^-(\KL_n)$ and let $y \in
\Gal(\T_n/\KL_n)$ be some lift of $\varphi(x)$.  If $\id{N}(y)$ fixes
$\KH(\K_n)$ and acts non trivially in $\E'_n/\KH(\K_n)$, then
$\beta(x) \neq 0$. Indeed, by choosing $\eu{A} \in x$ a prime with
$\lchooses{\T_n/\K_n}{\eu{A}} = y$, we see that $I := \id{N}(\eu{A})$
must be a principal ideal, since $\lchooses{\T_n)/\K_n}{I}$ fixes
$\KH(\K_n)$ and moreover, the Artin symbol acts non trivially on
$\E'_n/\KH(\K_n)$, so $\pi_N(I) \neq 0$. The claim follows from Lemma
\ref{princid}. 

Let $\T = \cup_n \T_n$.  In the projective limit, we
conclude that $x = (x_n)_{n \in \N}$ has $\beta(x) \neq 0$ iff for any
lift $y \in \Gal(\T/\KL_{\infty})$ of $\varphi(x) \in
\Gal(\KH(\KL_{\infty}))$ the norm $\id{N}(y)$ fixes $\KH(\K_{\infty})$
and acts non trivially in $\E'/\KH(K_{\infty})$, with $\E = \cup_n
\E'_n$. Moreover, $\T/\KH(\KL_{\infty})$ is the product of $| \Delta_m
|$ independent $\Z_p$-extensions and $\id{N}(y)$ is of $\lambda$-type.
We have thus from \ref{h1gen} a further isomorphism,
\begin{eqnarray}
 \label{h1e}
 \wh{H}^1(F, A^-(\KL_n)) \cong \Gal(\E'_n/\KH(\K_n)) \cong \Gal(\E'_{\infty}/\KH(\K_{\infty}).
\end{eqnarray}

\begin{figure}
\centering
\begin{tikzpicture}[node distance=2cm, auto] 
	\node(Q){$\mathbb{Q}$};
	\node(K)[above of=Q, right of=Q, node distance=2cm] {$\K$};
	\node(L)[right of=K]{$\mathbb{L}$};
	\node(Kn)[above of=K]{$\K_n$};
	\node(Ln)[above of=L]{$\mathbb{L}_n$};
	\node(Hnp)[above of=Kn, left of=Kn]{$\mathbb{H}_n'$};
	\node(bul)[above of=Hnp]{$\bullet$};
	\node(Enp)[left of=Hnp]{$\mathbb{E}_n'$};
	\node(Hn)[above of=Enp]{$\mathbb{H}_n$};
	\node(Tnp)[left of=Enp]{$\mathbb{T}_n'$};
	\node(Tn)[above of=Tnp]{$\mathbb{T}_n$};
	\node(Kinf)[above of=Kn, above of=Kn, above of=Kn]{$\K_{\infty}$};
	\node(Linf)[above of=Ln, above of=Ln, above of=Ln]{$\mathbb{L}_{\infty}$};
	\draw[-] (Q) to node {} (K);
	\draw[-] (K) to node {} (L);
	\draw[-] (K) to node {} (Kn);
	\draw[-] (L) to node {} (Ln);
	\draw[-] (Kn) to node {} (Hnp);
	\draw[-] (Ln) to node {} (bul);
	\draw[red] (Hnp) to node {$p$} (Enp);
	\draw[-] (Enp) to node {} (Tnp);
	\draw[red, dashed] (Hnp) to node {} (bul);
	\draw[red, dashed] (Enp) to node {} (Hn);
	\draw[-] (Tnp) to node {} (Tn);
	\draw[-] (bul) to node {$p$} (Hn);
	\draw[-] (Hn) to node {} (Tn);
	\draw[-] (Kn) to node {} (Kinf);
	\draw[-] (Ln) to node {} (Linf);
\end{tikzpicture}
\end{figure}


We have thus proved:
\begin{fact}
\label{h1str}
For every $n > m$, a class $x \in A^-(\KL_n)$ has non trivial image
$\beta(x) \in \wh{H}^1(F, A^-(\KL_n))$ iff for any lift $y \in
\Gal(\T_n/\KL_n)$ of $\varphi(x) \in \Gal(\KH(\KL_n)/\KL_n)$, the norm
$y' := \id{N}(y) \in \Gal(\T'_n/\K_n)$ fixes $\KH_n$ and acts non
trivially in $\E'_n$; equivalently, $y'$ generates a maximal cycle in
$\Gal(\T'_n/\KH(\K_n))$.  In the projective limit, $x = (x_n)_{n \in
  \N}$ has $\beta(x) \neq 0$ iff for any lift $y \in
\Gal(\T/\KL_{\infty})$ of $\varphi(x) \in \Gal(\KH(\KL_{\infty}))$ the
norm $\id{N}(y)$ fixes $\KH(\K_{\infty})$ and acts non trivially in
$\E'/\KH(K_{\infty})$, with $\E = \cup_n \E'_n$.  Moreover, there is
an exact sequence
\begin{eqnarray}
\label{eex}
 1 \ra \Gal(\KH(\KL_{\infty}/\KL_{\infty}) \ra \Gal(\T/\KL_{\infty}) \ra (\Z_p)^{|\Delta_m|} \ra 1, 
\end{eqnarray}
and thus $\Gal(\T/\K_{\infty})$ is Noetherian $\Lambda$-module.
\end{fact}

We now prove the Lemma \ref{hord}:
\begin{proof}
  Assume that $\wh{H}^1\left( F, \left( \id{M}(\apr{\KL}) \right)
  \right) \neq 0$. Consider the modules $M_n = \Gal(\T_n/\KL_n)$ and
  $M = \varprojlim_n M_n = \Gal(\T/\KL_{\infty})$. It follows from
  Fact \ref{h1str} that $M/\Gal(\KH(\KL_{\infty})) \cong
  \Z_p^{|\Delta_m|}$, so $M$ is a Noetherian $\Lambda$-module and the
  exact sequence \rf{eex} shows that $M$ is a rigid module and the
  further premises of Proposition \ref{tpdeco} hold too, as a
  consequnece of the choice of $\K$. Let $v' \in M$ be such that the
  restriction $v = v' \big \vert_{\KH(\KL_{\infty})} \in
  \id{M}(\apr{\KL})$ and it has non trivial image in $\wh{H}^1\left(
    F, \left( \id{M}(\apr{\KL}) \right) \right)$, via the inverse Artin map.
  In particular $v \not \in \id{L}(\apr{\KL})$ and also $v' \not \in \id{L}(M)$. Assume
  that $\ord(v) \leq p \exp(\id{M}(\apr{\K}))$.  The Proposition
  \ref{tpdeco} together with Remark \ref{kp1} imply that $T^2 v' =
  v'_{\mu} + v'_{\lambda}$ is decomposed.  But then $v'_{\mu} \in
  \Gal(\KH(\KL_{\infty}))$ so it follows from the above Fact that
  $\beta(T^2 v) = 0$. Thus, if $v \in \id{M}(\apr{\KL})$ has order 
  $\ord(v) \leq p \cdot \exp(\id{M}(\apr{\K})$, 
  then $T^2 v \in s \id{M}(\apr{\KL})$, which completes the proof of
  Lemma \ref{hord}.
\end{proof}
\subsection{Proof of the Proposition \ref{tpdeco}}
The proof of the Proposition requires a longer analysis of the growth
of modules $\Lambda x_n$ for indecomposed elements. This will be
divided in a sequence of definitions and Lemmata which eventually lead
to the proof.

The arguments of this section will take repeatedly advantage of the
following elementary Lemma\footnote{I owe to Cornelius Greither
  several elegant ideas which helped simplify my original proof.}:
\begin{lemma}
\label{ab}
Let $A$ and $B$ be finitely generated abelian $p-$groups denoted
additively, and let $N: B\ra A$, $\iota:A\ra B$ be two $\Z_p$ - linear
maps such that:
\begin{itemize}
\item[1.] $N$ is surjective.
\item[2.] The $p-$ranks $\prk(A) = \prk(p A) = \prk(B) = r$.
\item[3.]  $N(\iota(a))=p a,\forall a\in A$.
\end{itemize}
Then
\begin{itemize}
\item[ A. ] The inclusion $\iota(A) \subset p B$ holds
  unconditionally.
\item[ B. ] We have $\iota(A) = p B$ and $B[ p ] = \Ker(N) \subset
  \iota(A)$. Moreover, $\ord(x) = p \cdot \ord( \iota(N (x) ) )$ for all $x \in
  B$.
\item[ C. ] If there is a group homomorphism $T : B
  \ra B$ with $\iota(A) \subseteq \Ker(T)$ and $\nu := \iota \circ N =
  p + \binom{p}{2} T + O(T^2)$, then $\nu = \cdot p$, i.e. $\iota
  (N(x)) = p x$ for all $x \in B$.
\end{itemize}
\end{lemma}
\begin{proof}
  Since $A$ and $B$ have the same $p$-rank and $N$ is surjective, we
  know that the map $\overline{N} : B/p B \ra A/p A$ is an
  isomorphism\footnote{For finite abelian $p$-groups $X$ we denote
    $R(X) = X/pX$ by \textit{roof} of $X$ and $S(X) = X[ p ]$ is its
    \textit{socle}}. Therefore, the map induced by $N \iota$ on the
  roof is trivial. Hence $\overline{\iota} : A/p A \ra B/p B$ is also
  zero and thus $\iota(A) \subset p B$, which confirms the claim A.

  The premise $\prk(p A) = \prk(A)$ implies that $\prk(A)
  = \prk(\iota(A))$, as follows from
  \[ \prk(A) \geq \prk(\iota(A) = \prk(\iota(A)[ p ]) \geq \prk( p A )
  = \prk(A). \] We now consider the map $\iota' : A/p A \ra p B/ p^2
  B$ together with $\overline{N}$. From the hypotheses we know that $N
  \iota'$ is the multiplication by $p$ isomorphism: $\cdot p : A/ p A
  \ra p A/p^2 A$, as consequence of $\prk(A) = \prk(\iota(A)) = \prk(p
  A)$. It follows that $\iota'$ is an isomorphism of $\F_p$-vector
  spaces and hence $\iota : A \ra p B$ is surjective, so $\iota(A) = p
  B$. Consequently $| B | / |\iota(A) | = p^r$.  Let $x \in \Ker(N)$;
  since $N : B/p B \ra A/p A$ is surjective, the norm does not vanish
  for $x \not \in p B$.  Consequently, for $x \in \Ker(N) \subset p B$
  we have $N x = p x = 0$, so $\Ker(N) = B[ p ] \subset \iota(A)$, so
  $\Ker(N) = \iota(A)[ p ]$, as claimed.

  We now prove that $\ord(x) = p \cdot \ord(\iota(N (x) )$ for all $x
  \in B$. Consider the following maps $\pi : B \ra \iota(A), x \mapsto
  p x$ and $\pi' = \iota \circ N$.  Since $p B = \iota(A)$, both maps
  are surjective and there is an isomorphism $\phi : p B \ra p B$ such
  that $\pi = \phi \circ \pi'$.  Therefore
  \[ \ord(x)/p = \ord(p x) = \ord(\phi(p x)) = \ord(\iota(N(x))), \]
  and thus $\ord(x) = p \cdot \ord(\iota(N(x)))$, as claimed.

  For point C. we let $x \in B$, so $p x \in p B = \iota(A)$ and thus
  $p T x = T ( p x ) = 0$. Consequently $T x \in B[ p ] \subset
  \iota(A)$ and thus $T^2 x = 0$. From the definition of $\nu = \iota
  \circ N = p + Tp \frac{p-1}{2} + O(T^2)$ we conclude that $\nu x = p
  x + \frac{p-1}{2} T p x + O(T^2) x = p x$, which confirms the claim
  C. and completes the proof.
\end{proof}

In this section $\M_{\infty}/\M$ is an arbitrary $\Z_p$-extension in
which all the primes that ramify, ramify completely and $X' \subseteq
Y := \Gal(\T/\M_{\infty})$ is a Noetherian $\Lambda$-submodule,
the limit of the ray class groups $\Gal(\T_n/\M_n)$ to some ray module
associated to the base field with trivial finite part. Thus $\omega_n
x = 0$ implies $T x = 0$ by Fact \ref{sa}. The ray class groups $Y_n =
\varphi^{-1}(\Gal(\KH(\T_n)/\M_{\infty})$ may also coincide with class
groups $\A(\M_n)$.

We shall apply the Theorem \ref{tpdeco} to the concrete cases in which 
$\M = \KL$ and $Y$ is either $\Gal(\KH^-(\KL)/\KL_{\infty})$ or
$\Gal(\T/\M_{\infty})$, where $\T$ is the injective limit of 
subfields of ray class groups, defined above. 

We denote by $L, M, D$ the $\lambda$-, the $\mu$- and the
decomposed parts, respectively, of $X'$: thus, in the notation of the
introduction, we have $L = \id{L}(X'), M = \id{M}(X'), D =
\id{D}(X')$. One can for instance think of $\M$ as $\K$ in \S 2 and
$X' = \varphi(\apr{\K}) = X^-$ or $X' = \Gal(\T/\KH(\K_{\infty}))$.

For $x \in M$, the order is naturally defined by $\ord(x) = \min\{ p^k
\ : \ p^k x = 0 \}$.  Since $X'$ has no finite submodules, it follows
that $\ord(x) = \ord(T^j x)$ for all $j > 0$.

We introduce some distances $d_n : X' \times X' \ra \N$ as follows:
let $x, z \in X'$; then
\[ d_n(x, z) := \prk(\Lambda (x_n - z_n)); \quad d_n(x) = \prk(\Lambda
x_n) . \] We obviously have $d_n(x,z) \leq d_n(x,y)+d_n(x,z)$ and
$d_n(x) \geq 0$ with $d_n(x) = 0$ for the trivial module. Also, if $f
\in \Z_p[ T ]$ is some distinguished polynomial of degree $\phi =
\deg(f)$, then $d_n(x)-\phi \leq d_n(f x) \leq d_n(x)$ for all $x \in
X$. We shall write $d(x, y) = \lim_n d_n(x,y)$. Also, for explicit
elements $u, v \in X'_k$, we may write $d(u, v) = \prk(\Lambda
(u-v))$.  This can be useful for instance when no explicit lifts of
$u, v$ to $X'$ are known. The simplest fact about the distance is:
\begin{fact}
\label{dlam}
Let $x, z \in X'$ be such that $d_n(x, z) \leq N$ for some fixed bound
$N$ and all $n > 0$. Then $x-z \in L$ and $N \leq \ell := \prk(L)$.
For every fixed $d \geq \prk(L)$ there is an integer $n_0(d)$ such
that for any $x \in X' \setminus L$ and $n > n_0$, if $d_n(x) \leq d$
then $x \in \nu_{n,n_0} X'$.
\end{fact}
\begin{proof}
  The element $y = x-z$ generates at finite levels modules $\Lambda
  y_n$ of bounded rank, so it is neither of $\mu$-type nor
  indecomposed. Thus $y \in L$ and consequently $d_n(y) \leq \prk(L_n)
  \leq \ell$ for all $n$, which confirms the first claim.
  
  For the second claim, note that if $x \not \in L$, then $d_n(x) \ra
  \infty$, so the boundedness of $d_n(x)$ becomes a strong constraint
  for large $n$. Next we recall that $F(T) x \in M$ and since
  $d_n(F(T) x) \leq d_n(x)$, we may assume that $x \in M$. Now $d_n(x)
  \leq d$ implies the existence of some distinguished polynomial $h
  \in \Z_p[ T ]$ with $\deg(h) = d$ and such that $h(T) x_n = 0$. The
  exponent of $M$ is bounded by $p^{\mu}$, so there is a finite set
  $\id{H} \subset \Z_p[ T ]$ from which $h$ can take its values. Let
  now $n_0$ be chosen such that $\nu_{n_0,1} \in (h(T), p^{\mu})$ for
  all $h \in \id{H}$. Such a choice is always possible, since
  $\nu_{n,1} = p^{\mu} \cdot V(T) + T^{p^{n-\mu}-1} \cdot W(T)$ for
  some $V(T), W(T) \in \Lambda$. We may thus choose $n$ sufficiently
  large, such that the Euclidean division $T^{p^{n-\mu}} = q(T) h(T) +
  r(T)$ yields remainders $r(T)$ which are divisible by $p^{\mu}$ for
  all $h \in \id{H}$.  Let $n_0$ be the smallest such integer.  With
  this choice, for any $h \in \id{H}$, it follows that $h(T) x_n = 0$
  implies $\nu_{n_0,1} x_n = 0$ and thus $\nu_{n_0,1} x = \nu_{n,1} w$
  for some $w \in X'$. Then $\nu_{n_0,1}( x - \nu_{n,n_0} w) = 0$ and
  thus $x = \nu_{n,n_0} w$, by Fact \ref{sa} and the assumption on
  $X'$.
  
%

\end{proof}

We pass now to the proof of Proposition \ref{tpdeco}. 
\begin{proof}
  Let $x \in X'$ and suppose that $l$ is the smallest integer such
  that $p^l x \in L$ and let $f_x(T)$ be the minimal annihilator
  polynomial of $p^l x$, so $y := f_x(T) x \in M$, since $p^l y =
  0$. We claim that
  \[ p^j x_{n+j} - \iota_{n,n+j}(x_n) \in f_x(T) \Lambda x_{n+j}
  \subset \Lambda y_{n+j}, \quad \forall j > 0, \] and in particular
  $\iota_{n,n+l}(x_n) = p^l x_{n+l} - h_{n+l}(T) (f_x(T) x_{n+l})$ is
  decomposed. We let $n_1 > n_0$ be such that for all $x = (x_n)_{n
    \in \N}$ in $X' \setminus p X'$ we have $\ord(p^{l+\mu} x_n) > p$;
  here $n_0$ is the constant established in Fact \ref{dlam} with
  respect to the bound rank $d = \prk(L) + 1$.  For $n > n_1$ and $x
  \in X' \setminus ( D + p X')$, we have
\begin{eqnarray}
\label{strong}
p^l x_{n+l} - \iota_{n,n+l}(x_n) = f_x(T) h_n(T) x_{n+l} \in M_{n+l}.
\end{eqnarray} 
Indeed, consider the modules $B = \Lambda x_{n+1}/(f_x(T) \cdot
\Lambda x_{n+1})$ and $A = \Lambda x_n/ (f_x(T) \cdot \Lambda
x_n)$. Since $\iota_{n,n+1} x_{n} \not \in f_x(T) \Lambda x_{n+1}$ for
$n > n_1$ -- as follows from the condition imposed on the orders --
the induced map $\iota : A \ra B$ is rank preserving.  We can thus
apply the Lemma \ref{ab}, an $l$-fold iteration of which implies the
claim \rf{strong}.  We now apply the hypothesis that $p x = c + u \in
D$ for some $c \in L, u \in M[ p^{l-1} ]$; note that the condition
\rf{cstab} allows us to conclude from Lemma \ref{ab} and $c \in L$
that $\iota_{n,n+k}(c_n) = p^k c_{n+k}$ for all $c = (c_n)_{n \in \N}
\in L$ In particular,
\[ p \omega_{n} c_{n+1} = \omega_n \iota_{n,n+1}(c_n) = 0, \quad \hbox{so} \quad
\omega_{n+1} c_{n+1} \in L_{n+1}[ p ] = \iota_{1,n+1}(L_1[ p ]). 
\]
as a consequence of the same Lemma. We deduce under the above
hypothesis on $n$, that
\[ p^l x_{n+l} = p^{l-1} c_{n+l} = \iota_{n+1,n+l} c_{n+1} =
\iota_{n,n+l}(x_n) + h y_{n+l} . \] By applying $\omega_n$ to this
identity and using $\omega_n c_{n+1} \in \iota_{1,n+1}( L_1[ p ] )$,
so $T \omega_n c_{n+1} = 0$, we find $T h(T) \omega_n y_{n+l} =
0$. The Iwasawa's Theorem VI \cite{Iw} implies that there is some
$z(n) \in X'$ such that $T h \omega_n y = \nu_{n+l,1} z(n)$. Since
$p^l y = 0$, we have in addition $p^l \nu_{n+l,1} z(n) = 0$, so $p^l
z(n) = 0$ and $z(n) \in M$, by Fact \ref{nonu}.  We stress here the
dependency of $z \in X'$ on the choice of $n > n_1$ by writing $z(n)$;
this is however a norm coherent sequence and $z_m(n)$ will denote its
projection in $X'_m$ for all $m > 0$. We obtain $\nu_{n,1}( T^2 h(T) y
- \nu_{n+l,n} z(n) ) = 0$.  This implies $T^2 h(T) y = \nu_{n+l,n}
z(n)$, by the same Fact \ref{nonu}.  Reinserting this relation in the
initial identity, we find
  \begin{eqnarray}
  \label{cdeco}
   \iota_{n,n+l}( T^2 x_n - z_n(n)) = \iota_{n+1,n+l} (T^2 c_{n+1}). 
  \end{eqnarray} 
  
  We prove that \rf{cdeco} implies that $T^2 x$ must be decomposed.
  For this we invoke the Fact \ref{dlam} with respect to the sequence
  $w^{(n)} = T^2 x - z(n)$. We have $d_n( \iota_{n+1,n+l} (T^2
  c_{n+1}) ) \leq \prk(L)$ for all $n$; since $z(n) \in M$, and $T^2
  x$ is assumed indecomposed, then $w^{(n)} \not \in D$ either, and
  the Fact \ref{dlam} together with the choice of $n_1$ imply that
  $w_n^{(n)} \in \nu_{n,n_0} X'$ for all $n > n_1$. But then
\[ w_n^{(n)} = \iota_{n,n+l}( T^2 x_n -  z_n(n)) = \nu_{n,n_0} a_n  
\in \iota_{n_0,n}(X'_{n_0}) . \]
It follows in particular that 
\[ \ord(T^2 x_n - z_n(n)) \leq p^l \ord(\iota_{n,n+l}( T^2 x_n -
z_n(n))) \leq p^l \exp(X'_{n_0}) . \] This holds for arbitrary large
$n$, and since $z_n(n) \in M_n$, we have $\ord(T^2 x_n - z_n(n)) =
\ord(T^2 x_n)$ for $n > n_1$.  Therefore, the assumption that $T x
\not \in D$ implies that $\ord(T^2 x_n) \leq p^l \exp(X'_{n_0})$ for
all $n > n_1$: it is thus uniformly bounded for all $n$, which would
imply that $x \in M$, in contradiction with the assumption $x \not \in
D$. We have thus proved the claim for all $x \in X' \setminus (D + p
X')$. The general case follows by applying Nakayama to the module
$X'$.
 \end{proof}

 \textbf{Acknowledgments}: The main proof idea was first investigated
 in the wish to present the result in the memorial volume \cite{Alf}
 for Alf van der Poorten. However not more than the trivial case was
 correctly proved; S\"oren Kleine undertook within part of his PhD
 Thesis the task of brushing up the proof. He succeeded to do so, with
 the exception of the details related to decomposition and Lemma \ref{hord}, which we
 treated here. I owe to S\"oren Kleine for numerous useful discussions
 sustained during the development of this paper. I am particularly
 grateful to Andreas Nickel for critical discussions of an earlier
 version of this work. In G\"ottingen, the graduate students in winter
 term of 2015-2016, Pavel Coupek, Vlad Cri\c{s}an and Katharina
 M\"uller presented in a seminar this work, I owe them for their great
 contribution to the verification of the final redaction. The paper is dedicated to the memory of
 Alf van der Poorten.

\bibliographystyle{abbrv} 
\bibliography{muNull}
\end{document}